\documentclass{article}
\usepackage {amssymb}
\usepackage {amsmath}
\usepackage {amsmath}
\usepackage {amsthm}
\usepackage{graphicx}
\usepackage {amscd}
\usepackage[colorlinks,linkcolor=black, anchorcolor=black, citecolor=red]{hyperref}
\newtheorem{thm}{Theorem}
\newtheorem{conj}{Conjecture}
\newtheorem{prop}{Proposition}
\newtheorem{defn}{Definition}
\newtheorem{rem}{Remark}
\newtheorem{exmp}{Example}
\newtheorem{lem}{Lemma}

\begin{document}
\title{\textbf{Remarks on logarithmic K-stability}}
\author{Chi Li}
\date{}
\maketitle \vspace*{-30pt}
\begin{abstract}
\noindent ABSTRACT: We make some observation on the logarithmic version of K-stability.
\end{abstract}

\section{Introduction}

Let $(X,J)$ be a Fano manifold, that is, $K_X^{-1}$ is ample. The basic problem in K\"{a}hler geometry is to determine whether $(X, J)$ has a 
K\"{a}hler-Einstein metric (cf. \cite{T}) 

On way to attack this problem is to use continuity method.  Fix a reference K\"{a}hler metric ${{\omega}}\in c_1(X)$. Its Ricci curvature
$Ric({{\omega}})$ also lies in $c_1(X)$. So there exists $h_{{\omega}}\in C^{\infty}(X)$ such that
\[
Ric({{\omega}})-{{\omega}}=\partial\bar{\partial}h_{{\omega}},\quad \int_X e^{h_{{\omega}}}{{\omega}}^n=\int_X{{\omega}}^n
\]
Consider the following family of Monge-Amp\`{e}re equations.
\begin{align}
({{\omega}}+\partial\bar{\partial}\phi_t)^n=e^{h_{{\omega}}-t\phi}{{\omega}}^n\tag*{$(*)_t$}\label{CMAt}
\end{align}
This is equivalent to the equation for K\"{a}hler forms:
\begin{equation}\label{cont1}
Ric(\omega_\phi)=t\omega_\phi+(1-t)\omega
\end{equation}

By Yau's theorem \cite{Yau}, we can always solve \ref{CMAt} for $t=0$. If we could solve \ref{CMAt} for $t=1$, we would get 
K\"{a}hler-Einstein metric.  However, it was first showed by Tian \cite{T2} that we may not be able to
solve \ref{CMAt} on certain Fano manifold for $t$ sufficiently close
to 1. Equivalently, for such a Fano manifold, there is some $t_0<1$,
such that there is no K\"{a}hler metric $\omega$ in $c_1(X)$ which
can have $Ric(\omega)\ge t_0\omega$.  

The existence problem of K\"{a}hler-Einstein metric is a special case of the existence problem of constant scalar curvature K\"{a}hler (cscK) metric.  For the latter, we fix an ample line bundle $L$ on $(X, J)$. 
We have the following folklore conjecture.
For the definition of K-stability, see \cite{T97}, \cite{Dol02} or Definition \ref{dflgkstb}. 

\begin{conj}[Tian-Yau-Donaldson](\cite{T97},\cite{Dol02})
There is a smooth constant scalar curvature K\"{a}hler metric in $c_1(L)$ on $(X,J)$ if and only if $(X,J, L)$ is K-stable.
\end{conj}

Return to the continuity method \ref{CMAt} and let $R(X)=\sup \{t: \mbox{\ref{CMAt} is solvable }\}$. Sz\'{e}kelyhidi proved that
\begin{prop}[\cite{S}]
\[
R(X)=\sup\{t: \exists \mbox{ a K\"{a}hler metric }\; {{\omega}}\in
c_1(X) \mbox{ such that } Ric({{\omega}})> t{{\omega}}\}
\]
\end{prop}
In particular, $R(X)$ is independent of reference metric ${{\omega}}$. 

There is another continuity method we can try. 
Let $Y\in |-K_X|$ be a general element, then $Y$ is a smooth Calabi-Yau hypersurface. The K\"{a}hler-Einstein metric with cone singularity 
along $Y$ of cone angle $2\pi\beta$ is a solution to the following distributional equation
\begin{equation}\label{cnkeq}
Ric(\omega)=\beta\omega+(1-\beta)\{Y\}
\end{equation}
\begin{conj}[Donaldson]\label{conj1}
There is a cone-singularity solution $\omega_\beta$ to \eqref{cnkeq} for any parameter $\beta\in (0, R(X))$. If $R(X)<1$, there is no solution for parameter
$\beta \in (R(X),1)$. 
\end{conj}

The purpose of this note is to discuss the logarithmic version of K-stability and prove the following result.

\begin{thm}\label{main}
Let $X_\triangle$ be a toric Fano variety with a $(\mathbb{C}^*)^n$ action. Let $Y$ be a general hyperplane section of $X_\triangle$. 
When $\beta<R(X_\triangle)$, $(X_\triangle, \beta Y)$ is log-K-stable along any 1 parameter subgroup in $(\mathbb{C}^*)^n$. When $\beta=R(X_\triangle)$, $(X_\triangle, \beta Y)$ is semi-log-K-stable along any 1 parameter subgroup in $(\mathbb{C}^*)^n$ and there is a 1 parameter subgroup in $(\mathbb{C}^*)^n$ which has vanishing log-Futaki invariant. When $\beta>R(X_\triangle)$, $(X_\triangle,\beta Y)$ is not log-K-stable.
\end{thm}

This explains and generalizes slightly the calculation in \cite{Dol11} and gives some evidence for the Conjecture \ref{conj1} (Combined with Conjecture \ref{logTYD}).

We prove the above result by calculating $R(X_\triangle)$ and log-Futaki invariant explicitly. $R(X_\triangle)$ was calculated in \cite{Li10} based on 
Wang-Zhu's work \cite{WZ}. The main formula for log-Futaki invariant is \eqref{toricftk}.

A toric Fano
manifold $X_\triangle$ is determined by a reflexive lattice polytope
$\triangle$ (For details on toric manifolds, see \cite{Oda}). For
example, let $Bl_p\mathbb{P}^2$ denote the manifold obtained by blowing up one point on $\mathbb{P}^2$. Then $Bl_p\mathbb{P}^2$ is a toric Fano manifold and is determined by the following polytope.

Any such polytope $\triangle$ contains the origin
$O\in\mathbb{R}^n$. We denote the barycenter of $\triangle$ by
$P_c$. If $P_c\neq O$, the ray
$P_c+\mathbb{R}_{\ge0}\cdot\overrightarrow{P_c O}$ intersects the
boundary $\partial\triangle$ at point $Q$.

\begin{thm}\cite{Li10}\label{thm1}
If $P_c\neq O$,
\[
R(X_\triangle)=\frac{\left|\overline{OQ}\right|}{\left|\overline{P_cQ}\right|}
\]
Here $\left|\overline{OQ}\right|$, $\left|\overline{P_cQ}\right|$ are lengths of line
segments $\overline{OQ}$ and $\overline{P_cQ}$. In other words, 
\[
Q=-\frac{R(X_\triangle)}{1-R(X_\triangle)}P_c \in \partial\triangle
\]

If $P_c=O$, then
there is K\"{a}hler-Einstein metric on $X_\triangle$ and
$R(X_\triangle)=1$.
\end{thm}

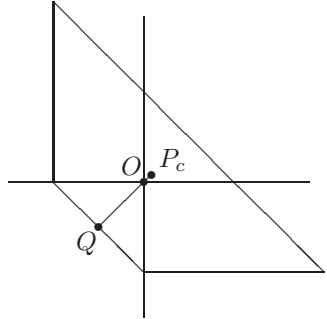
\vspace*{-20mm}
\begin{figure}[h]
 \begin{center}\label{figure1}
\setlength{\unitlength}{1mm}
\begin{picture}(50,50)
\put(-18, 0){\line(1,0){40}} \put(0, -18){\line(0,1){40}}
\put(-12,0){\line(0,1){24}} \put(-12,0){\line(1,-1){12}}
\put(-12,24){\line(1,-1){36}} \put(0,-12){\line(1,0){24}}
\put(-6,-6){\line(1,1){7}} \put(0,0){\circle*{1}}
\put(1,1){\circle*{1}} \put(-6,-6){\circle*{1}} \put(-9,-9){$Q$}
\put(-3,1){$O$} \put(2,2){$P_c$}
\end{picture}
\end{center}
\end{figure}

\vspace*{20mm}

\textbf{Acknowledgement:} I would like to thank Professor Gang Tian for helpful discussion and constant encouragement. In particular, he suggested the log-K-stability and told me
the result in \cite{LT92} which provides the starting example of the logarithmic version of Tian-Yau-Donaldson conjecture.

\section{Log-Futaki invariant}
In this section, we recall Donaldson's definition of log-Futaki invariant \eqref{analyticftk}.
Let $(X,L)$ be a polarized projective variety and $D$ be a normal crossing divisor:
\[
D=\sum_{i=1}^r \alpha_i D_i
\]
with $\alpha_i\in (0,1)$.

From now on, we fix a Hermitian metric $\|\cdot\|_i=h_i$ and defining section $s_i$ of the line bundle $[D_i]$. 


Assume $\omega\in c_1(L)$ is a smooth K\"{a}hler form. We define
\[
\overline{\mathcal{P}}({\omega})=\left\{{\omega}_{\phi}:={\omega}+\frac{\sqrt{-1}}{2\pi}\partial\bar{\partial}\phi;\; 
\phi\in L^{\infty}(X)\cap C^{\infty}(X\backslash D) \mbox{ such that } \omega+\frac{\sqrt{-1}}{2\pi}\partial\bar{\partial}\phi\ge 0\right\}
\]

Around any point $p\in X$, we can find local coordinate $\{z_i; i=1,\cdots, n\}$, such that $D$ is defined by 
\[
D=\cup_{i=1}^{r_p}\alpha_i\{z_i=0\}
\]
where $r_p=\sharp\{i; p\in D_i\}$.

\begin{defn}
We say that $\hat{\omega}\in \overline{\mathcal{P}}(\omega)$ is a conic K\"{a}hler metric on $(X, D)$,  if around $p$,
$\omega$ is quasi-isometric to the metric
\[
\sum_{i=1}^{r_p}\frac{dz_i\wedge d\bar{z}_i}{|z_i|^{2\alpha_i}}+\sum_{j=r_p+1}^ndz_j\wedge d\bar{z}_j
\]
We will simply say that $\hat{\omega}$ is a conic metric if it's clear what $D$ is.
\end{defn}

Geometrically, this means the Riemannian metric determined by $\omega$ has conic singularity along each $D_i$ of conic angle $2\pi(1-\alpha_i)$.
 
\begin{rem}
Construction of K\"{a}hler-Einstein metrics with conic singularites was proposed long time ago by Tian, see \cite{Tconic} in which he used such metrics to prove inequalities of Chern numbers in algebraic geometry. 
\end{rem}
 
One consequence of this definition is that globally the volume form has the form
\[
\hat{\omega}^n=\frac{\Omega}{\prod_{i=1}^{r}\|s_i\|_i^{2\alpha_i}}
\]
where $\Omega$ is a smooth volume form. For any volume form $\Omega$,  let $Ric(\Omega)$ denote the curvature of the Hermitian metric on $K_X^{-1}$ determined by $\Omega$. Then, by abuse of notation,
\begin{eqnarray}\label{ricdistr}
Ric(\hat{\omega})&=&Ric(\hat{\omega}^n)=Ric(\Omega)+\frac{\sqrt{-1}}{2\pi}\sum_{i=1}^r \alpha_i\partial\bar{\partial}\log\|s_i\|_i^2=Ric(\Omega)-\sum_{i=1}^r \alpha_i c_1([D_i], h_i)+\sum_{i=1}^r
\alpha_i\{D_i\}\nonumber\\
&=&Ric(\Omega)-c_1([D], h)+\{D\}
\end{eqnarray}
where $h=\otimes_{i=1}^rh_i^{\alpha_i}$ and 
$s=\otimes_{i=1}^rs_i^{\alpha_i}$ are Hermitian metric and defining section of 
the $\mathbb{R}$-line bundle $[D]=\otimes_{i=1}^r[D_i]^{\alpha_i}$.

Here we used the Poinc\'{a}re-Lelong identity:
\[
\frac{\sqrt{-1}}{2\pi}\partial\bar{\partial}\log\|s_i\|_i^2=-c_1([D_i],h_i)+\{D_i\}
\]
where $\{D_i\}$ is the current of integration along the divisor $D_i$.

The scalar curvature of $\hat{\omega}$ on its smooth locus $X\backslash D$ is 
\[
S(\hat{\omega})=\hat{g}^{i\bar{j}}\hat{R}_{i\bar{j}}=\frac{nRic(\hat{\omega})\wedge\hat{\omega}^{n-1}}{\hat{\omega}^n}=\frac{n(Ric(\Omega)-c_1([D],h))
\wedge\hat{\omega}^{n-1}}{\hat{\omega}^n}
\]

So if $S(\hat{\omega})$ is constant, then the constant only depends on cohomological classes by the identity:
\begin{equation}\label{cncst}
n\mu_1:=\frac{n(c_1(X)-c_1([D]))\wedge [c_1(L)]^{n-1}}{c_1(L)^n}=\frac{-n(K_X+D)\cdot L^{n-1}}{L^n}=n\mu-\frac{Vol(D)}{Vol(X)}
\end{equation}

Here
\[
n\mu=\frac{n\; c_1(X)\cdot c_1(L)^{n-1}}{c_1(L)^{n}}=\frac{-nK_X\cdot L^{n-1}}{L^{n}}
\]is the average scalar curvature for smooth K\"{a}hler form in $c_1(L)$. And

\[
Vol(D)=\int_D \frac{c_1(L)^{n-1}}{(n-1)!}=\frac{D\cdot L^{n-1}}{(n-1)!},\quad Vol(X)=\int_X\frac{c_1(L)^n}{n!}=\frac{L^n}{n!}
\]

Now assume $\mathbb{C}^*$ acts on $(X, L)$ and $v$ is the generating holomorphic vector field.
Recall that the ordinary Futaki-Calabi invariant (\cite{Fut}, \cite{Cal}) is defined by
\[
F(c_1(L))(v)=-\int_X\theta_v(S(\omega)-n\mu)\frac{\omega^n}{n!}
\]
where $\theta_v$ satisfies
\[
\iota_v\omega=\bar{\partial}\theta_v
\]

Now assume ${\hat{\omega}_\infty}\in\overline{\mathcal{P}}(\omega)$ is a conic metric and satisfies

\begin{equation}
S({\hat{\omega}_\infty})=n\mu_1
\end{equation}

Assume $D$ is preserved by the $\mathbb{C}^*$ action. Let's calculate the ordinary Futaki invariant using the conic metric ${\hat{\omega}_\infty}$. 
Let $\hat{\theta}_v=\hat{\theta}({\hat{\omega}_\infty}, v)$. Then near $p\in D$, $v\sim \sum_{i=1}^{r_p}c_iz_i\partial_{z_i}+\tilde{v}$ with $\tilde{v}=o(z_1\cdots z_{r_p})$ holomorphic. 
$\hat{\theta}_v\sim \sum_{i=1}^{r_p}|z_i|^{2(1-\alpha_i)}$. 

We then make use of the distributional identity \eqref{ricdistr} to get
\begin{eqnarray*}
F(c_1(L))(v)&=&-\int_{X}\hat{\theta}_v(n Ric({\hat{\omega}_\infty})-n\mu{\hat{\omega}_\infty})\wedge\frac{{\hat{\omega}_\infty}^{n-1}}{n!}\\
&=&
-\int_{X}\hat{\theta}_v\left[(nRic(\Omega)-nc_1([D],h)-n\mu_1{\hat{\omega}_\infty})+
n\{D\}-(n\mu-n\mu_1){\hat{\omega}_\infty}\right]\wedge\frac{{\hat{\omega}_\infty}^{n-1}}{n!}\\
&=&-\int_X\hat{\theta}_v(S({\hat{\omega}_\infty})-n\mu_1)\frac{{\hat{\omega}_\infty}^n}{n!}-\int_X\{D\}\hat{\theta}_v\frac{{\hat{\omega}_\infty}^{n-1}}{(n-1)!}+(n\mu-n\mu_1)\int_X\hat{\theta}_v\frac{{\hat{\omega}_\infty}^n}{n!}\\
&=&-\left(\int_D\hat{\theta}_v\frac{{\hat{\omega}_\infty}^{n-1}}{(n-1)!}-\frac{Vol(D)}{Vol(X)}\int_X\hat{\theta}_v\frac{{\hat{\omega}_\infty}^n}{n!}\right)
\end{eqnarray*}

So we get

\[
0=F(c_1(L))(v)+\left(\int_D\hat{\theta}_v\frac{{\hat{\omega}_\infty}^{n-1}}{(n-1)!}-\frac{Vol(D)}{Vol(X)}\int_X\hat{\theta}_v\frac{{\hat{\omega}_\infty}^n}{n!}\right)
\]

Since the two integrals in the above formula is integration of (singular) equivariant forms, they are independent of the chosen K\"{a}hler metric in
$\overline{\mathcal{P}}(\omega)$ with at worst conic singularities. In particular, we can choose the smooth K\"{a}hler metric $\omega$, then we just discover the log-Futaki invariant defined by Donaldson:

\begin{defn}\cite{Dol11}
\begin{equation}\label{analyticftk}
F(c_1(L), D)(v)=F(c_1(L))(v)+\left(\int_D\theta_v\frac{\omega^{n-1}}{(n-1)!}-\frac{Vol(D)}{Vol(X)}\int_X\theta_v\frac{\omega^n}{n!}\right)
\end{equation}
\end{defn}

\begin{rem}
This differs from the formula in \cite{Dol11} by a sign. And we think of $D$ as a cycle with real coefficients, so if we replace $D$ by $(1-\beta)\triangle$, we have the same formua as that in \cite{Dol11}.
\end{rem}

\section{log-K-energy and Berman's formulation}

We can integrate the log-Futaki-invariant to get log-K-energy 
\begin{eqnarray}\label{lgkeng}
\nu_{{\omega}, D}(\phi)&=&-\int_0^1 dt\int_X(S({\omega}_t)-\underline{S})\dot{\phi}\frac{{\omega}_t^n}{n!}+\int_0^1dt\int_D \dot{\phi}\frac{{\omega}_t^{n-1}}{(n-1)!}-
\frac{Vol(D)}{Vol(X)}\int_0^1dt\int_X\dot{\phi}\frac{{\omega}_t^n}{n!}\nonumber\\
&=&\nu_{{\omega}}(\phi)+\int_0^1\int_X\left(\frac{\sqrt{-1}}{2\pi}\partial\bar{\partial}\log\|s\|^2+c_1([D], h)\right) \dot{\phi}\frac{{\omega}_t^{n-1}}{(n-1)!}+\frac{Vol(D)}{Vol(X)}F^0_{{\omega}}(\phi)\nonumber\\
&=&\nu_{{\omega}}(\phi)+\frac{Vol(D)}{Vol(X)}
F^0_{{\omega}}(\phi)+\mathcal{J}^{\chi_D}_{{\omega}}(\phi)+\int_X\log\|s_D\|^2 \frac{{{\omega}}_\phi^n}{n!}
\end{eqnarray}
where $\chi_D=c_1([D], h)$ is the Chern curvature form. The functionals $F^0_{{\omega}}(\phi)$ and $\mathcal{J}^\chi_{{\omega}}(\phi)$ are defined by:
\[
F^0_{{\omega}}(\phi)=-\int_0^1dt\int_X\dot{\phi}\frac{{\omega}_{\phi_t}^n}{n!}
\]

\[
\mathcal{J}^\chi_{{\omega}}(\phi)=\int_0^1 dt\int_X \dot{\phi}\chi\wedge\frac{{\omega}_{\phi_t}^{n-1}}{(n-1)!}
\]

Let's now focus on the Fano case as in the beginning of this paper.  \eqref{cnkeq} is equivalent to the following singular complex Monge-Amp\`{e}re equation:
\begin{equation}\label{sgma}
({\omega}+\partial\bar{\partial}\phi)^n=e^{-\beta\phi}\frac{\Omega_1}{\|s\|^{2(1-\beta)}}
\end{equation}
with $\Omega_1=e^{h_{{\omega}}}{\omega}^n$ and $s$ is a defining section of $[Y]$. Note that the line bundle $[Y]=K_X^{-1}$ has the Hermitian metric $\|\cdot\|$ such that the curvature is $\omega$. 

We have $D=(1-\beta)Y$. Since $[Y]=K_X^{-1}$, we can assume $\chi_D=(1-\beta){\omega}$, $Vol((1-\beta) D)=n(1-\beta) Vol(X)$.  Then \eqref{lgkeng} becomes
\begin{eqnarray*}
\nu_{{\omega}, D}&=&\nu_{{\omega}}(\phi)+(1-\beta)\left(n F_{{\omega}}^0(\phi)+ \mathcal{J}^{{\omega}}_{{\omega}}(\phi)\right)+(1-\beta)\int_X\log\|s\|^2\frac{{\omega}_\phi^n}{n!}\\
&=&\nu_{{\omega}}(\phi)+(1-\beta)(I_{{\omega}}-J_{{\omega}})+(1-\beta)\int_X\log\|s\|^2\frac{{\omega}_\phi^n}{n!}\\
&=&\int_X\log\frac{{\omega}_\phi^n}{\Omega_1}-\beta(I_{{\omega}}-J_{{\omega}})+(1-\beta)\int_X\log\|s\|^2\frac{{\omega}_\phi^n}{n!}
\end{eqnarray*}

We have used the well known formula for K-energy \cite{T}:
\[
\nu_{{\omega}}(\phi)=\int_X\log\frac{{\omega}_\phi^n}{\Omega_1}-(I_{{\omega}}-J_{{\omega}})(\phi)
\]
where
\[
I_{{\omega}}(\phi)=\int_X\phi({\omega}^n-{\omega}_\phi^n)/n!
\]
\[
J_{{\omega}}(\phi)=F_{{\omega}}^0(\phi)+\int_X\phi\frac{{\omega}^n}{n!}
\]
And it's easy to verify that
\[
n F_{{\omega}}^0(\phi)+\mathcal{J}^{{\omega}}_{{\omega}}(\phi)=(I_{{\omega}}-J_{{\omega}})(\phi)=-\left(\int_X\phi{\omega}_\phi^n+F^0_{{\omega}}(\phi)\right)
\]
From above formula, we  see that, in Fano case, the log-K-energy coincides with Berman's free energy associated with \eqref{sgma} (\cite{Ber})
\[
\nu_{\omega, D}=\int_X\log\frac{{\omega}_\phi^n}{\Omega_1/\|s\|^{2(1-\beta)}}\frac{{\omega}_\phi^n}{n!}+\beta\left(\int_X\phi{\omega}_\phi^n+F_{{\omega}}^0(\phi)\right)
\]

\section{Log-K-stability}
We imitate the definition of K-stability to define log-K-stability. First we recall the definition of test configuration \cite{Dol02} or special degeneration \cite{T97} of a polarized projective variety $(X, L)$. 

\begin{defn}
A test configuration of $(X, L)$,  consists of
\begin{enumerate}
\item a scheme $\mathcal{X}$ with a $\mathbb{C}^*$-action;
\item a $\mathbb{C}^*$-equivariant line bundle $\mathcal{L}\rightarrow\mathcal{X}$
\item a flat $\mathbb{C}^*$-equivariant map $\pi: \mathcal{X}\rightarrow\mathcal{C}$, where $\mathbb{C}^*$ acts on $\mathbb{C}$ by
multiplication in the standard way;
\end{enumerate}
such that any fibre $X_t=\pi^{-1}(t)$ for $t\neq 0$ is isomorphic to $X$ and $(X,L)$ is isomorphic to $(X_t,\mathcal{L}|_{X_t})$.
\end{defn}

Any test configuration can be equivariantly embedded into $\mathbb{P}^N\times\mathbb{C}^*$ where the $\mathbb{C}^*$ action on $\mathbb{P}^N$ is given by a 1 parameter subgroup of $SL(N+1,\mathbb{C})$. If $Y$ is any subvariety of $X$, the test configuration of $(X, L)$ also induces a
test configuration $(\mathcal{Y}, \mathcal{L}|_{\mathcal{Y}})$ of $(Y, L|_Y)$ .

Let $d_k$, $\tilde{d}_k$ be the dimensions of $H^0(X, L^{k})$, $H^0(Y, L|_Y^{\;k})$, and $w_k$, $\tilde{w}_k$ be the weights of $\mathbb{C}^*$ action on $H^0(X_0, \mathcal{L}|_{X_0}^{\;k})$, $H^0(Y_0, \mathcal{L}|_{Y_0}^{\;k})$, respectively. Then we have expansions:
\[
w_k=a_0k^{n+1}+a_1k^n+O(k^{n-1}),\quad
d_k=b_0k^n+b_1k^{n-1}+O(k^{n-2})
\]

\[
\tilde{w}_k=\tilde{a}_0k^n+O(k^{n-1}), \quad
\tilde{d}_k=\tilde{b}_0k^{n-1}+O(k^{n-2})
\]

If the central fibre $X_0$ is smooth, we can use equivariant differential forms to calculate the coefficients by \cite{Dol02}. Let ${\omega}$ be a smooth K\"{a}hler form in $c_1(L)$, and $\theta_v=\mathcal{L}_v-\nabla_v$,  then
\begin{equation}\label{coef1}
a_0=-\int_X\theta_v\frac{{\omega}^n}{n!};\; a_1=-\frac{1}{2}\int_X\theta_v S({\omega})\frac{{\omega}^n}{n!}
\end{equation}
\begin{equation}\label{coef2}
b_0=\int_X\frac{{\omega^n}}{n!}=Vol(X);\; b_1=\frac{1}{2}\int_X S({\omega})\frac{{\omega}^n}{n!}
\end{equation}
\begin{equation}\label{coef3}
\tilde{a}_0=-\int_{Y_0}\theta_v\frac{{\omega}^{n-1}}{(n-1)!};\; \tilde{b}_0=\int_{Y_0}\frac{{\omega}^{n-1}}{(n-1)!}=Vol(Y_0)
\end{equation}
\begin{rem}
To see the signs of coefficients and give an example, we consider the case where $X=\mathbb{P}^1$, $L=\mathcal{O}_{\mathbb{P}^1}(k)$. $\mathbb{C}^*$ acts on 
$\mathbb{P}^1$ by multiplication: $t\cdot z=tz$. A general $D\in |L|$ consists of $k$ points. As $t\rightarrow 0$, $t\cdot D\rightarrow k\{0\}$. $D$ is
the zero set of a general degree $k$ homogeneous polynomial $P_k(z_0, z_1)$ and $k\{0\}$ is the zero set of $z_1^k$. $\mathbb{C}^*$ acts
on $H^0(\mathbb{P}^1, \mathcal{O}(k))$ by $t\cdot z_0^iz_1^j=t^{-j}z_0^iz_1^j$ so that $\lim_{t\rightarrow 0}[t\cdot P_k(z_0, z_1)]=[z_1^k]$, where 
$[P_k]\in \mathbb{P}(H^0(\mathbb{P}^1, \mathcal{O}(k)))$. Take the Fubini-Study metric $\omega_{FS}=\frac{\sqrt{-1}}{2\pi}\partial\bar{\partial}\log(1+|z|^2)=\frac{\sqrt{-1}}{2\pi}\frac{dz\wedge d\bar{z}}{(1+|z|^2)^2}$, then $\theta_v=\frac{\partial\log(1+|z|^2)}{\partial \log |z|^2}=\frac{|z|^2}{1+|z|^2}$. So
\[
-a_0=\int_{\mathbb{P}^1}\theta_v\omega_{FS}=\int_0^{+\infty}\frac{r^2}{(1+r^2)^3}2rdr=\frac{1}{2}
\]
\[
-a_1=\frac{1}{2}\int_{\mathbb{P}^1}S(\omega_{FS})\theta_v\omega_{FS}=\int_{\mathbb{P}^1}\theta_v\omega_{FS}=\frac{1}{2}
\]
While
\[
w_k=-(1+\cdots+k)=-\frac{1}{2}k^2-\frac{1}{2}k
\]
which gives exactly $a_0=a_1=-\frac{1}{2}$.
\end{rem}

Comparing \eqref{analyticftk}, \eqref{coef1}-\eqref{coef3}, we can define the algebraic log-Futaki invariant of the given test configuration to be
\begin{eqnarray}\label{logftk}
F(\mathcal{X}, \mathcal{Y}, \mathcal{L})&=&\frac{2(a_1b_0-a_0b_1)}{b_0}+(-\tilde{a}_0+\frac{\tilde{b}_0}{b_0}a_0)\nonumber\\
&=&\frac{(2a_1-\tilde{a}_0)b_0-a_0(2b_1-\tilde{b}_0)}{b_0}
\end{eqnarray}

\begin{defn}\label{dflgkstb}
$(X,Y, L)$ is log-K-stable along the test configuration $(\mathcal{X}, \mathcal{L})$ if $F(\mathcal{X}, \mathcal{Y}, \mathcal{L})\le0$, and equality holds
if and only if $(\mathcal{X}, \mathcal{Y}, \mathcal{L})$ is a product configuration. 

$(X, Y, L)$ is semi-log-K-stable along $(\mathcal{X}, \mathcal{L})$ if $F(\mathcal{X}, \mathcal{Y}, \mathcal{L})\le0$. Otherwise, it's unstable.

$(X, Y, L)$ is log-K-stable (semi-log-K-stable) if, for any integer $r>0$, $(X, Y, L^r)$ is log-K-stable (semi-log-K-stable) along any test configuration of $(X, Y, L^r)$. 

\end{defn}

\begin{rem}
When $Y$ is empty, then definition of log-K-stability becomes the definition of K-stability. (\cite{T97}, \cite{Dol02})
\end{rem}

\begin{rem}\label{infsgn}
In applications, we sometimes meet the following situation. Let $\lambda(t): \mathbb{C}^*\rightarrow SL(N+1, \mathbb{C})$ be a 1 parameter subgroup. As 
$t\rightarrow \infty$, $\lambda(t)$ will move $X, Y\subset\mathbb{P}^N$ to the limit scheme $X_0$, $Y_0$. Then stability condition is equivalent to the other opposite sign condition $F(X_0, Y_0, v)\ge 0$. This is of course related to the above definition by transformation $t\rightarrow t^{-1}$. 
\end{rem}

\begin{exmp}[Orbifold]
Assume $X$ is smooth. $Y=\sum_{i=1}^r(1-\frac{1}{n_i})D_i$ is a normal crossing divisor, where $n_i>0$ are integers. The conic K\"{a}hler metric on
$(X, Y)$ is just the orbifold K\"{a}hler metric on the orbifold $(X, Y)$. Orbifold behaves similarly as smooth variety,  but in the calculation, we need to use orbifold canonical bundle $K_{orb}=K_X+Y$. For example, think $L$ as an orbifold line bundle on $X$, then the orbifold Riemann-Roch says that
\begin{eqnarray*}
dim H^0_{orb}((X, Y), L)&=&\frac{L^n}{n!}k^n+\frac{1}{2}\frac{-(K_X+Y)\cdot L^n}{(n-1)!}k^{n-1}+O(k^{n-2})\\
&=&b_0k^n+\frac{1}{2}(2b_1-\tilde{b}_0)k^{n-1}+O(k^{n-2})
\end{eqnarray*}
For the $\mathbb{C}^*$-weight of $H^0_{orb}((X, Y), L)$, we have expansion:
\begin{equation*}
w^{orb}_k=a_0^{orb}k^{n+1}+a_1^{orb}k^{n}+O(k^{n-1})
\end{equation*}
By orbifold equivariant Riemann-Roch, we have the formula:
\[
a_0^{orb}=\int_{X}\hat{\theta}_v\frac{\hat{\omega}^n}{n!}=\int_X\theta_v\frac{\omega^n}{n!}=a_0
\]
\[
a_1^{orb}=\int_{X}\hat{\theta}_v S(\hat{\omega})\frac{\hat{\omega}^n}{n!}
\]
To calculate the second coefficient $a_1^{orb}$, we choose an orbifold metric $\hat{\omega}$, then by \eqref{coef1}:
\begin{eqnarray*}
a_1&=&-\frac{1}{2}\int_X\hat{\theta}_v n\; Ric(\hat{\omega})\wedge\frac{\hat{\omega}^{n-1}}{n!}\\
&=&-\frac{1}{2}\int_X \hat{\theta}_v n(Ric(\Omega)-c_1([D], h)+\{D\})\wedge \frac{\hat{\omega}^{n-1}}{n!}\\
&=&-\frac{1}{2}\int_X \hat{\theta}_vS(\hat{\omega})\frac{\hat{\omega}^n}{n!}-\frac{1}{2}\int_D \hat{\theta}_v\frac{\hat{\omega}^{n-1}}{(n-1)!}\\
&=&a_1^{orb}-\frac{1}{2}\int_D\theta_v\frac{\omega^{n-1}}{(n-1!}=a_1^{orb}+\frac{1}{2}\tilde{a}_0
\end{eqnarray*}
So
\begin{equation}
a_1^{orb}=\frac{1}{2}(2a_1-\tilde{a}_0)
\end{equation}
Comparing \eqref{logftk}, we see that the log-Futaki invariant recovers the orbifold Futaki invariant, and similarly log-K-stability recovers orbifold K-stability. Orbifold
Futaki and orbifold K-stability were studied by Ross-Thomas \cite{RT09}. 

\end{exmp}

\begin{exmp}
$X=\mathbb{P}^1$, $L=K_{\mathbb{P}^1}^{-1}=\mathcal{O}_{\mathbb{P}^1}(2)$, $Y=\sum_{i=1}^{r}\alpha_i p_i$.  For any $i\in\{1,\cdots, r\}$, we choose the coordinate $z$ on $\mathbb{P}^1$, such
that $z(p_i)=0$. Then consider the holomorphic vector field $v=z\partial_{z}$. $v$ generates the 1 parameter subgroup $\lambda(t)$ :
$\lambda(t)\cdot z=t\cdot z$. As $t\rightarrow \infty$, $\lambda(t)$ degenerate $(X, Y)$ into the pair $(\mathbb{P}^1, \alpha_i \{0\}+\sum_{j\neq i}\alpha_j
\{\infty\})$.  We take $\theta_v=\frac{-|z|^{-2}+|z|^2}{|z|^{-2}+1+|z|^2}$. Then it's easy to get the log-Futaki invariant of the degeneration determined by $\lambda$:
\[
F(\mathbb{P}^1, \sum_{i=1}^r\alpha_ip_i, \mathcal{O}_{\mathbb{P}^1}(2))(\lambda)=\sum_{j\neq i}\alpha_j-\alpha_i
\]
If $(\mathbb{P}^1, \sum_{i=1}^r\alpha_ip_i)$ is log-K-stable, by Remark \ref{infsgn}, we have
\begin{equation}\label{relalp}
\sum_{j\neq i}\alpha_j-\alpha_i>0
\end{equation}
Equivalently, if we let $t\rightarrow 0$, we get $\alpha_i-\sum_{j\neq i}\alpha_j<0$ from log-K-stability.

Let's consider the problem of constructing singular Riemannian metric $g$ of constant scalar curvature on $\mathbb{P}^1$ which has conic angle $2\pi(1-\alpha_i)$ at $p_i$ and is smooth elsewhere. Assume $p_i\neq\infty$ for any $i=1,\dots, r$. Under conformal coordinate $z$ of $\mathbb{C}\subset\mathbb{P}^1$, $g=e^{2u}|dz|^2$. $u$ is a smooth function in the punctured complex plane $\mathbb{C}-\{p_1,\dots, p_r\}$ so that near each $p_i$, $u(z)=-2\alpha_i\log|z-p_i|+$a continuous function, where $\alpha_i\in (0,1)$ and $u=-2\log|z|+$ a continuous function near infinity. We call such function is of conic type. The condition of constant scalar curvature
corresponds to the following Liouville equations.
\begin{enumerate}
\item
$
\Delta u=-e^{2u}
$
\item
$
\Delta u=0
$
\item
$
\Delta u=e^{2u}
$
\end{enumerate}
which correspond to scalar curvature=1, 0, -1 case respectively.

For such equations, we have the following nice theorem due to Troyanov, McOwen, Thurston, Luo-Tian.
\begin{thm}[See \cite{LT92} and the reference there]
\begin{enumerate}
\item 
For equation 1, it has a solution of conic type if and only if
\begin{enumerate}
\item[(a)] $\sum_{i=1}^r\alpha_i<2$, and
\item[(b)] $\sum_{j\neq i}\alpha_j-\alpha_i>0$, for all $i=1, \dots, n$. 
\end{enumerate}
\item
For equation 2,  it has a solution of conic type if and only if  (a): $\sum_{i=1}^r\alpha_i=2$. 

In this case, (a) implies the condition: (b) $\sum_{j\neq i}\alpha_j-\alpha_i>0$, for all $i=1, \dots, r$. 

\item
For equation 3,  it has a solution of conic type if and only if  (a): $\sum_{i=1}^r\alpha_i>2$. 

Again in this case, (a) implies the condition: (b) $\sum_{j\neq i}\alpha_j-\alpha_i>0$, for all $i=1, \dots, r$. 

\end{enumerate}
Moreover, the above solutions are all unique.
\end{thm}

Note that $\deg (-(K_{\mathbb{P}^1}+\sum_{i=1}^{r}\alpha_ip_i))=2-\sum_{i=1}^r\alpha_i$, so by \eqref{cncst}, conditions (a) in above theorem correspond to the cohomological conditions for the scalar curvature to be positive, zero, negative respectively. 
While the condition (b) is the same as \eqref{relalp}. So by the above theorem, if $(\mathbb{P}^1, \sum_{i=1}^r\alpha_ip_i)$ is log-K-stable, then there is 
a conic metric on $(\mathbb{P}^1, \sum_{i=1}^r\alpha_ip_i)$ with constant curvature whose sign is the same as that of $2-\sum_i\alpha_i$.
\end{exmp}

This example clearly suggests 
\begin{conj}[Logarithmic version of Tian-Yau-Donaldson conjecture]\label{logTYD}
There is a constant scalar curvature conic K\"{a}hler metric on $(X, Y)$ if and only if $(X, Y)$ is log-K-stable. 
\end{conj}

\section{Toric Fano case}

\subsection{Log-Futaki invariant for 1psg on toric Fano variety}

For a reflexive
lattice polytope $\triangle$ in
$\mathbb{R}^n=\Lambda\otimes_\mathbb{Z}\mathbb{R}$, we have a Fano
toric manifold $(\mathbb{C}^*)^n\subset X_\triangle$ with a
$(\mathbb{C}^*)^n$ action. In the following, we will sometimes just write $X$ for $X_\triangle$ for simplicity.

Let $(S^1)^n\subset(\mathbb{C}^*)^n$ be the standard real maximal
torus. Let $\{z_i\}$ be the standard coordinates of the dense orbit
$(\mathbb{C}^*)^n$, and $x_i=\log|z_i|^2$. We have
\begin{lem}
Any $(S^1)^n$ invariant K\"{a}hler metric ${{\omega}}$ on $X$ has a
potential $u=u(x)$ on $(\mathbb{C}^*)^n$, i.e.
${{\omega}}=\frac{\sqrt{-1}}{2\pi}\partial\bar{\partial}u$. $u$ is a
proper convex function on $\mathbb{R}^n$, and satisfies the momentum
map condition:
\[
Du(\mathbb{R}^n)=\triangle
\]
Also,
\begin{equation}\label{toricvol}
\frac{(\partial\bar{\partial}u)^n/n!}{\frac{dz_1}{z_1}\wedge\frac{d\bar{z}_1}{\bar{z}_1}\cdots\wedge\frac{dz_n}{z_n}\wedge\frac{d\bar{z}_n}{\bar{z}_n}}=\det\left(\frac{\partial^2u}{\partial x_i\partial x_j}\right)
\end{equation}
\end{lem}

Let $\{p_\alpha;\; \alpha=1,\cdots, N\}$ be all the 
lattice points of $\triangle$. Each $p_\alpha$ corresponds to a
holomorphic section $s_\alpha\in H^0(X_\triangle,
K^{-1}_{X_\triangle})$. We can embed $X_\triangle$ into
$\mathbb{P}^{N}$ using $\{s_\alpha\}$. Define $u$ to be
the potential on $(\mathbb{C}^*)^n$ for the pull back of
Fubini-Study metric (i.e.
$\frac{\sqrt{-1}}{2\pi}\partial\bar{\partial}u={{\omega}}_{FS}$):
\begin{equation}\label{u0}
u=\log\left(\sum_{\alpha=1}^N e^{<p_\alpha,x>}\right)+C
\end{equation}
$C$ is some constant determined by normalization condition:
\begin{equation*}
\int_{\mathbb{R}^n}e^{-u}dx=Vol(\triangle)=\frac{1}{n!}\int_{X_\triangle}{{\omega}}^n=\frac{c_1(X_\triangle)^n}{n!}
\end{equation*}
By the above normalization of $u$, it's easy to see that
\begin{equation*}
e^{h_{{\omega}}}=\frac{|\cdot|_{FS}^2}{|\cdot|_{{\omega}^n}^2}=\frac{e^{-u}}{{{\omega}}^n/(\frac{dz_1}{z_1}\wedge\frac{d\bar{z}_1}{\bar{z}_1}\cdots\wedge\frac{dz_n}{z_n}\wedge\frac{d\bar{z}_n)}
{\bar{z}_n}}
\end{equation*}
So
\begin{equation}\label{torich}
h_{{\omega}}=-\log\det(u_{ij})-u
\end{equation}

Now let's calculate the log-Futaki invariant for any 1-parameter subgroup in $(\mathbb{C}^*)^n$. Each 1-parameter subgroup in $(\mathbb{C}^*)^n$ is determined by some $\lambda\in\mathbb{R}^n$ such that the generating holomorphic vector field is
\[
v_\lambda=\sum_{i=1}^n\lambda_i z_i\frac{\partial}{\partial z_i}
\]

A general Calabi-Yau hypersurface $Y\in |-K_X|$ is a hyperplane section given by the equation:
\[
s:=\sum_{\alpha=1}^Nb(p_\alpha)z^{p_\alpha}=0
\]
By abuse of notation, we denote $\lambda(t)$ to be the 1 parameter subgroup generated by $v_\lambda$, then
\begin{equation}\label{actpoly}
\lambda(t)\cdot s= \sum_{\alpha=1}^N b(p_\alpha) t^{-\langle p_\alpha, \lambda\rangle} z^{p_\alpha}
\end{equation}

Let 
\[
W(\lambda)=max_{p\in\triangle}\langle p, \lambda \rangle
\]
Then $H_\lambda=\{p\in\mathbb{R}^n, \langle p, \lambda\rangle=W(\lambda)\}$ is a supporting plane of $\triangle$, and
 \[
\mathcal{F}_\lambda:=\{p\in\triangle; \langle p, \lambda\rangle=W(\lambda)\}=H_\lambda\cap\triangle
\]
 is a face of $\triangle$. 

We have $\lim_{t\rightarrow 0} [s]=\left[s_0:=\sum_{p_\alpha\in\mathcal{F}_\lambda}b(p_\alpha)z^{p_\alpha}\right]$, and  
by \eqref{actpoly}, the $C^*$-weight of $s_0$ is
$-W(\lambda)$. 

\begin{prop}
Let $F(K_X^{-1}, \beta Y)(\lambda)$ denote the Futaki invariant of the test configuration associated with the 1 parameter subgroup
generated by $v_\lambda$. We have
\begin{equation}\label{toricftk}
F(K_X^{-1}, \beta Y)(\lambda)=-\left(\beta\langle P_c, \lambda\rangle+(1-\beta)W(\lambda)\right)Vol(\triangle)
\end{equation}
\end{prop} 

\begin{proof}
We will use the algebraic definition of log-Futaki invariant \eqref{logftk} to do the calculation.

Note that $(X, Y, K_X^{-1})$ degenerates to $(X, Y_0, K_X^{-1})$ under $\lambda$.

$Y_0$ is a hyperplane section of $X$, and $s_0\in H^0(X, K_X^{-1})$ is the defining section, i.e. 
$Y_0=\{s_0=0\}$. Then

\[
H^0(Y_0, K_X^{-1}|_{Y_0}^{\;k})\cong H^0(X, K_X^{-k})/(s_0\otimes H^0(X,K_X^{-(k-1)}))
\]
So
\[
\tilde{w}_k=w_k-(w_{k-1}-W(\lambda)d_{k-1})
\]
Plugging the expansions, we get
\[
\tilde{a}_0=(n+1)a_0+W(\lambda) b_0
\]
Note that $\tilde{b}_0=n b_0=n Vol(\triangle)$, we have
\[
-\tilde{a}_0+\frac{\tilde{b}_0}{b_0}a_0=-a_0-W(\lambda)b_0
\]
where
\[
-a_0=\int_X\theta_v\frac{{\omega}^n}{n!}=\int_{\mathbb{R}^n}\sum_i\lambda_iu_i\det(u_{ij})dx=\int_\triangle \sum_i\lambda_iy_idy=Vol(\triangle)
\langle P_c, \lambda\rangle
\]
By \eqref{torich}, the ordinary Futaki invariant is given by
\begin{eqnarray*}
F(c_1(X))(v_\lambda)&=&\int_X v(h_{{\omega}})\frac{{\omega}^n}{n!}=-\int_{\mathbb{R}^n}\sum_{i=1}^n\lambda_i\frac{\partial u}{\partial x_i}\det ({u}_{ij})dx\\
&=&-\int_\triangle \sum_{i}\lambda_iy_idy=-Vol(\triangle)\langle P_c,\lambda\rangle
\end{eqnarray*}

Substituting these into \eqref{logftk}, we get
\begin{eqnarray*}
F(K_X^{-1}, \beta Y)(\lambda)&=&-Vol(\triangle)\langle P_c, \lambda\rangle+(1-\beta)(Vol(\triangle)\langle P_c, \lambda\rangle-W(\lambda)Vol(\triangle))\\
&=&-(\beta\langle P_c, \lambda\rangle+(1-\beta)W(\lambda))Vol(\triangle)
\end{eqnarray*}
\end{proof}

\begin{proof}[Proof of Theorem \ref{main}]

Note that for any $P_\lambda\in\mathcal{F}_\lambda\subset\partial\triangle$, $W(\lambda)=\langle P_\lambda , \lambda\rangle$.
By Theorem \ref{thm1}, we have
\begin{eqnarray*}
F(K_X^{-1}, \beta Y)(\lambda)&=&\left(\frac{\beta}{1-\beta}\frac{1-R(X)}{R(X)} \langle Q, \lambda\rangle-W(\lambda)\right) (1-\beta)Vol(\triangle)\\
&=&\langle Q_\beta-P_\lambda, \lambda\rangle
\end{eqnarray*}
where $Q_\beta=\frac{\beta}{1-\beta}\frac{1-R(X)}{R(X)} Q$. 

Note that $\lambda$ is a outward normal vector of $H_\lambda$. By convexity of $\triangle$, it's easy to see that (see the picture after Example \ref{example2})
\begin{itemize}
\item
$\beta< R(X)$:  $Q_\beta\in \triangle^{\circ}$.  For any $\lambda\in\mathbb{R}^n$,
$
\langle Q_\beta-P_\lambda, \lambda\rangle<0
$.
\item
$\beta=R(X)$: $Q_\beta=Q\in \partial
\triangle$.  For any $\lambda\in\mathbb{R}^n$, 
$
\langle Q_\beta-P_\lambda, \lambda\rangle\le 0
$.
Equality holds if and only if  $\langle Q, \lambda\rangle=W(\lambda)$, i.e. $H_\lambda$ is a supporting plane of $\triangle$ at point $Q$.
\item
$\beta>R(X)$: $Q_\beta\notin\overline{\triangle}$. There exists $\lambda\in\mathbb{R}^n$ such that $\langle Q_\beta-P_\lambda, \lambda\rangle>0$

\end{itemize}
\end{proof}

\subsection{Example}
\begin{enumerate}
\item
$X_\triangle=Bl_p\mathbb{P}^2$. See the picture in Introduction. $P_c=\frac{1}{4}(\frac{1}{3}, \frac{1}{3})$, $Q=-6 P_c\in\partial\triangle$, so
$R(X)=\frac{6}{7}$. 

If we take $\lambda=\langle -1, -1\rangle$, then $W(\lambda)=1$. So by \eqref{toricftk}
\[
F(K_X^{-1}, \beta Y)(\lambda)=\frac{2}{3}\beta-4(1-\beta)
\]

So $F(K_X^{-1}, \beta Y)(\lambda)\le0$ if and only if $\beta\le\frac{6}{7}$, and equality holds exactly when $\beta=\frac{6}{7}$.
\item \label{example2}
$X_\triangle=Bl_{p,q}\mathbb{P}^2$,
$P_c=\frac{2}{7}(-\frac{1}{3},-\frac{1}{3})$,
$Q=-\frac{21}{4}P_c\in\partial\triangle$, so
$R(X_\triangle)=\frac{21}{25}$. 

If we take $\lambda_1=\langle 1,1\rangle$, then $W(\lambda_1)=1$. By \eqref{toricftk},
\[
F(K_X^{-1}, \beta Y)(\lambda_1)=\frac{2}{3}\beta -\frac{7}{2}(1-\beta)
\]
$F(K_X^{-1}, \beta Y)(\lambda_1)\le 0$ if and only if $\beta\le\frac{21}{25}$. 

This is essentially the same as Donaldson's calculation in \cite{Dol11}.

\newpage

If we take $\lambda_3=\langle -1, 2\rangle$, then $W(\lambda_3)=\langle -1, 2\rangle\cdot\langle -1,1\rangle=3$. By \eqref{toricftk}
\[
F(K_X^{-1}, \beta Y)(\lambda_3)=\frac{1}{3}\beta-\frac{21}{2}(1-\beta)
\]

So $F(K_X^{-1}, \beta Y)(\lambda_3)\le 0$ if and only if $\beta\le\frac{63}{65}$ which means that $(X, \beta Y)$ is log-K-stable along $\lambda_3$ when 
$\beta\le\frac{21}{25}<\frac{63}{65}$.

\end{enumerate}

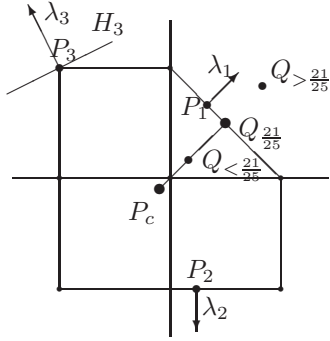
\vspace*{-25mm}
\begin{figure}[h] \label{pic}
\begin{center}
\setlength{\unitlength}{0.7mm}
\begin{picture}(80,70)
\put(-30,0){\line(1,0){60}} \put(0,-30){\line(0,1){60}}
\put(-21,-21){\line(0,1){42}} \put(-21,21){\line(1,0){21}}
\put(0,21){\line(1,-1){21}} \put(21,0){\line(0,-1){21}}
\put(-21,-21){\line(1,0){42}} \put(-2,-2){\line(1,1){12.5}}
\put(-2,-2){\circle*{2}} \put(10.5,10.5){\circle*{2}}
\put(-8,-8){$P_c$} \put(13,8){$Q_{\frac{21}{25}}$}

\put(-21,21){\line(2,1){10}} \put(-21,21){\line(-2,-1){10}} 
\put(-15,28){$H_3$}
\put(-23,23){$P_3$}
\put(-21,21){\vector(-1,2){6}}
\put(-24,30){$\lambda_3$}

\put(5, -21){\vector(0,-1){8}}
\put(6,-26){$\lambda_2$}
\put(3,-19){$P_2$}
\put(5,-21){\circle*{1.5}}

\put(7,14){\circle*{1.5}} \put(7,14){\vector(1,1){6}} 
\put(7 ,20){$\lambda_1$} \put(2,12){$P_1$}

\put(3.5,3.5){\circle*{1.5}} \put(6,2){$Q_{<\frac{21}{25}}$}  

\put(17.5,17.5){\circle*{1.5}} \put(19,19){$Q_{>\frac{21}{25}}$}  

\put(-21,21){\circle*{1.5}} \put(-21,-21){\circle*{1}} \put(21,-21){\circle*{1}}
\put(0,21){\circle*{1}}   \put(21,0){\circle*{1}}  \put(0,0){\circle*{1}} \put(-21,0){\circle*{1}}

\end{picture}
\end{center}
\end{figure}

Department of Mathematics, Princeton University, Princeton, NJ
08544, USA

E-mail address: chil@math.princeton.edu

\vspace*{20mm}

\begin{thebibliography}{99}


\bibitem{Ber} Berman, R.: A thermodynamic formalism for Monge-Amp\`{e}re euations, Moser-Trudinger inequalities and K\"{a}hler-Einstein
metrics \href{http://arxiv.org/abs/1011.3976}{arXiv:1011.3976}

\bibitem{Cal}Calabi, E.: Extremal K\"{a}hler metrics II, in "Differential geometry and complex analysis", Springer-Verlag, Berlin, Heidelberg, New York, 1985, pp. 95-114.

\bibitem{Dol02}Donaldson, S.K.: Scalar curvature and stability of toric varieties, Jour. Differential Geometry 62 289-349 (2002)

\bibitem{Dol11}Donaldson, S.K.: K\"{a}hler metrics with cone singularities along a divisor, \href{http://arxiv.org/abs/1102.1196}{arXiv:1102.1196}

\bibitem{Fut}Futaki, A.: An obstruction to the existence of Einstein K\"{a}hler metrics, Invent. Math., 73, 437-443 (1983)

\bibitem{Li10}Li, C.: Greatest lower bounds on Ricci curvature for toric Fano manifolds, \href{http://arxiv.org/abs/0909.3443}{arXiv:0909.3443}

\bibitem{LT92}Luo, F. and Tian, G.: Liouville equation and spherical convex polytopes, Proceedings of the American Mathematical Society, Vol 116, 
No. 4 (1992) 1119-1129

\bibitem{Oda}Oda, T.: Convex bodies and algebraic geometry-an
introduction to the theory of toric varieties, Springer-Vergla, 1988

\bibitem{RT09}Ross, J., and Thomas, R.: Weighted projective embeddings, stability of orbifolds and 
constant scalar curvature K\"{a}hler metrics\\
\href{http://arxiv.org/abs/0907.5214}{arXiv:0907.5214}

\bibitem{S}Sz\'{e}kelyhidi, G.: Greatest lower bounds on the Ricci
curvature of Fano manifolds,\\
\href{http://arxiv.org/abs/0903.5504}{arXiv:0903.5504}

\bibitem{Tconic}Tian, G.: K\"{a}hler-Einstein metrics on algebraic manifolds. Lecture notes in Mathematics, 
1996, Volume 1646/1996, 143-185. 

\bibitem{T}Tian, G.: Canonical Metrics on K\"{a}hler Manifolds,
Birkhauser, 1999
\bibitem{T2}Tian, G.: On stability of the tangent bundles of Fano
varieties. Internat. J. Math. 3, 3(1992), 401-413

\bibitem{T97}Tian, G.: K\"{a}hler-Einstein metrics with positive scalar curvature, Invent.
Math. 130 (1997), 1-39

\bibitem{WZ}Wang, X.J. and Zhu, X.H.: K\"{a}hler-Ricci solitons on toric
manifolds with positive first Chern class. Advances in Math. 188
(2004) 87-103
\bibitem{Yau}Yau, S.T.: On the Ricci curvature of a compact K\"{a}hler
manifold and the complex Monge-Amp\`{e}re equation, I, Comm. Pure
Appl. Math. 31 (1978) 339-441.
\end{thebibliography}
\end{document}